\documentclass[12pt]{article}
\usepackage{amsmath,mathdots,wasysym, amsthm,amssymb,MnSymbol,extsizes}
\usepackage[all]{xy}
\usepackage[active]{srcltx}

\usepackage[T1]{fontenc}

\newtheorem{theorem}{Theorem}[section]
\newtheorem{corollary}{Corollary}[section]
\newtheorem{lemma}{Lemma}[section]

\theoremstyle{remark}

\newcommand{\un}{\underline}

\newcommand{\as}%
{\text{\raisebox{0.75pt}
{$\scriptstyle\medstar$\!}}}

\newcommand{\ti}{\tilde}
\newcommand{\ba}{\tilde}

\newcommand{\mc}{\mathcal}
\newcommand{\cc}{\mathbb C}
\newcommand{\rr}{\mathbb R}
\newcommand{\ff}{\mathbb F}
\newcommand{\pp}{\mathbb P}

\DeclareMathOperator{\diag}{diag}

\newcommand{\ffff}{\ff^{\times}/\ff^{\times}_0}

\newcommand{\ind}%
{\mathop{\rm ind}\nolimits}
\newcommand{\End}%
{\mathop{\rm End}\nolimits}

\renewcommand{\le}{\leqslant}
\renewcommand{\ge}{\geqslant}
\newcommand{\mat}[1]{\begin{bmatrix}
 #1  \end{bmatrix}}

 \newcommand{\matt}[1]{\left[\begin{smallmatrix}\\
 #1 \end{smallmatrix}\right]}

\begin{document}
\title{Isometric and selfadjoint operators on a vector space with nondegenerate diagonalizable form\thanks{Published in: {\bf Linear Algebra Appl. 587 (2020) 92--110}.}}

\date{}

\author{Jonathan V. Caalim\thanks{University of the Philippines-Diliman, Quezon City, Philippines jcaalim@math.upd.edu.ph}
                       \and
Vyacheslav Futorny\thanks{Department of Mathematics, University of S\~ao Paulo, Brazil, {futorny@ime.usp.br}}
                      \and
Vladimir V.
Sergeichuk\thanks{Institute of Mathematics,
Kiev, Ukraine, {sergeich@imath.kiev.ua}}
                      \and
Yu-ichi Tanaka\thanks{Joso Gakuin High School, Tsuchiura City, Ibaraki Prefecture, Japan
mathlogic.ty@gmail.com
}}

\maketitle

\begin{abstract}
Let $V$ be a vector space over a field $\ff$ with scalar product given by a nondegenerate sesquilinear form whose matrix is diagonal in some basis. If $\ff=\cc$, then we give canonical matrices of isometric and selfadjoint operators on $V$ using known classifications of isometric and selfadjoint operators on a complex vector space with nondegenerate Hermitian form.
If $\ff$ is a field of characteristic different from $2$, then we give canonical matrices of isometric, selfadjoint, and skewadjoint operators on $V$ up to classification of symmetric and Hermitian forms over finite extensions of $\ff$.


\end{abstract}

\newpage

\section{Introduction}

Our main results are the following:
\begin{itemize}

  \item
In \emph{Theorem \ref{yyj}}, the problem of classifying isometric, selfadjoint, and skewadjoint operators on a vector space  with nondegenerate diagonalizable sesquilinear form over a field $\ff$ of characteristic different from $2$ with nonidentity involution is reduced to the problem of classifying such operators on a vector space  with nondegenerate Hermitian form over $\ff$. \emph{Corollary \ref{ynh1}} of Theorem \ref{yyj} gives canonical matrices of   isometric and selfadjoint operators on a complex vector space  with nondegenerate diagonalizable sesquilinear form.

  \item
In \cite[Theorem 6]{ser_izv},  selfadjoint operators
on a vector space  with nondegenerate symmetric or Hermitian form over a field $\ff$ of characteristic different from 2 are classified up to classification of symmetric and Hermitian forms over finite extensions of $\ff$.  \emph{Theorem \ref{Theorem 6}} is an extended version of \cite[Theorem 6]{ser_izv} that includes  skewadjoint operators. \emph{Corollary \ref{theor}} of Theorem \ref{Theorem 6} gives known canonical matrices of selfadjoint and skewadjoint operators on a complex or real vector space  with nondegenerate symmetric or Hermitian form.
\end{itemize}

We denote by $\ff$ a field of characteristic different from 2 with involution $a\mapsto \tilde a$; i.e., a bijection $\ff\to\ff$ that satisfies
$\widetilde{a+b}=\tilde a+\tilde b$, $\widetilde{ab}=\tilde a\tilde b$, and
$\tilde{\tilde a}=a$
for all $a,b\in\ff$.

We denote by $V$ a vector space over $\ff$ with nondegenerate sesquilinear form $\mc F:V\times V\to\ff$
that is semilinear in the first argument and linear in the second ($\mc F$  is bilinear if the involution $a\mapsto \tilde a$ is the identity).  We suppose that $\mc F$ is  \emph{diagonalizable}; i.e.,  its matrix is diagonal in some basis.

A linear operator $\mc A:V\to V$ is \emph{isometric} if $\mc F(\mc A u,\mc A v)=\mc F(u, v)$ for all $u,v\in V$,
 \emph{selfadjoint} if $
\mc F(u,\mc A v)=\mc F(\mc A u, v)$, and
 \emph{skewadjoint} if $
\mc F(u,\mc A v)=- \mc F(\mc A u, v)$.

If $(\mc A,\mc F)_V$ is a pair consisting of a linear operator and a sesquilinear form on $V$, and $(A,F)$ is its matrix pair  in some basis of $V$, then we can reduce it by transformations
\begin{equation*}\label{edc}
(A,F)\mapsto (S^{-1}AS,S^{\as}FS),\qquad
 S^{\as}:=\tilde A^T,\quad
S\text{ is nonsingular}
\end{equation*}
by changing the basis.
Write
\[
J_n(\lambda):=\mat{\lambda&1&&0\\
&\lambda&\ddots \\ &&\ddots&1\\0&&&\lambda}
\qquad(n\text{-by-}n,\  \lambda \in\ff).
\]
The \emph{direct sum} of matrix pairs $(A,B)$ and
$(A',B')$ is  $(A\oplus A',B\oplus B')$.

The following corollary is proved in Section \ref{vvz} using \cite[Theorem 5]{ser_izv} and Theorem \ref{Theorem 6}.

\begin{corollary}[of Theorem \ref{yyj}]\label{ynh1}
Let $\mc A:V\to V$ be a linear operator on a complex vector space $V$ with scalar product given by a nondegenerate diagonalizable sesquilinear form $\mc F:V\times V\to\cc$.

{\rm(a)} If $\mc A$ is isometric, then there exists a basis of\/ $V$ in which the pair $(\mc A,\mc F)_V$ is given by a direct sum, determined uniquely up to permutation of summands, of matrix pairs of two types:
\begin{itemize}
  \item[\rm(i)]
$\left(\lambda \matt{1&2&\cdots&2
\\&1&\ddots&\vdots
\\&&\ddots&2
\\0&&&1},
\mu i^{n-1}\matt{0&&&& \udots
\\&&&\!\!-1\!\!
\\&&1&
\\&\!\!-1\!\!&\\
1&&&&0
}\right)$,
in which the matrices are $n$-by-$n$,  $\lambda,\mu \in\cc$, and
$|\lambda|=|\mu |= 1$;

  \item[\rm(ii)]
$\left(\mat{J_n(\lambda)&0\\0&
J_n(\lambda)^{-*}},
\mu\mat{0&I_n\\
I_n&0}\right)$, in which  $0\ne\lambda \in\cc$,
$|\lambda|\ne 1$, and  $\lambda $ is determined up to
replacement by
$\bar\lambda^{-1}$; $\mu\in\cc$, $|\mu |= 1$, and $\mu$ is determined up to
replacement by
$-\mu$.
\end{itemize}

{\rm(b)}  If $\mc A$ is selfadjoint, then there exists a basis of\/ $V$ in which the pair $(\mc A,\mc F)_V$ is given by a direct sum, determined uniquely up to permutation of summands, of matrix pairs of two types:
\begin{itemize}
  \item[\rm(i)]
$\left(J_n(\lambda ),\mu  \matt{0&&1
\\&\udots&\\
1&&0\\[1pt]}\right)$,  in which $\lambda \in\rr$, $\mu \in\cc$,  and $|\mu |= 1$;

  \item[\rm(ii)]
$\left(\mat{J_n(\lambda)&0\\0&
J_n(\lambda)^*},
\mu\mat{
0& I_n\\I_n&0
}\right)$, in which  $\lambda \in\cc\setminus\rr$ and  $\lambda $ is determined up to
replacement by
$\bar\lambda$; $\mu\in\cc$, $|\mu |= 1$, and $\mu$ is determined up to
replacement by
$-\mu$.
\end{itemize}
\end{corollary}

Canonical forms of isometric, selfadjoint, and skewadjoint operators on a complex or real vector space with nondegenerate symmetric or Hermitian form are given in \cite{au,dok1,goh,goh-R,hig,mehl,mehl1,meh,rod}.

By an \emph{$\varepsilon $-Hermitian form} with $\varepsilon \in\{1,-1\}$ we mean a Hermitian form if $\varepsilon =1$ and a skew-Hermitian form if $\varepsilon =-1$. By a \emph{$\zeta$-adjoint operator} with $\zeta\in\{1,-1\}$ we mean a selfadjoint operator if $\zeta =1$ and a skewadjoint operator if $\zeta =-1$.

In Corollary \ref{theor}, we give canonical matrices of $\zeta $-adjoint operators on a vector
space $V$
with nondegenerate
$\varepsilon$-Hermitian form over the following fields:
\begin{itemize}
  \item[\rm(A)]
An algebraically
closed field $\ff$ of
characteristic
different from 2
with the identity
involution.

  \item[\rm(B)] An algebraically
      closed field $\ff$ with
      nonidentity involution $a\mapsto \tilde a$.  By \cite[Lemma 2.1(b)]{ser_iso}, the characteristic of $\ff$ is 0, the field $\pp:=\{a\in\ff\,|\,a=\tilde a\}$ is a \emph{real closed field}  (i.e., a field
whose algebraic closure has
degree $2$), $\ff=\pp+i\pp$ with $i^2=-1$, and the involution on $\ff$ is
$\alpha +\beta i\mapsto \alpha -\beta i$  with $\alpha,\beta \in\pp$.

  \item[\rm(C)]
A real closed field
$\mathbb P$. By
\cite[Lemma 2.1(a)]{ser_iso}, there exists a unique
linear ordering $\le$ on $\pp$
such that $a>0$ and
$b>0$ imply $a+b>0$ and
$ab>0$; its
algebraic closure is
 $\mathbb
P+\mathbb P i$  with $i^2=-1$ and with
 the
involution
$a+bi\mapsto a-bi$.
\end{itemize}

If $\ff$ is (B), then it is sufficient to study the case
\begin{equation}\label{cfm}
\varepsilon =\zeta =1
\end{equation}
since if $\mc F$ is a skew-Hermitian form, then $i\mc F$ is Hermitian, and if $\mc A$ is a skewadjoint operators on $V$, then $i\mc A$ is a selfadjoint operator:
\[\mc F(u,i\mc Av)=i\mc F(u,\mc Av)=-i\mc F(\mc Au,v)=-\mc F(\ba{\imath}\mc Au,v)=\mc F(i\mc Au,v).
\]

Define the $n\times n$ matrices
\begin{equation}\label{mmx}
Z_n:=\mat{0&&&1\\&&\udots\\&1\\1&&&0},
\quad
Z_n(\zeta)=\mat{\\[-22pt]0&&&&\iddots\\&&&\zeta\\&&1\\
&\zeta\\ 1&&&&0\\[2pt]}  \text{ with } \zeta\in\{1,-1\},
\end{equation}
and the $2n\times 2n$ matrices $K_{n}(\varepsilon)$ with $\varepsilon \in\{1,-1\}$:
\[
{\arraycolsep=0pt
K_{n}(\varepsilon) :=\mat{\\[-20pt]0&&&&\iddots
\\&&&\!\!\!\!\!-K(\varepsilon)\!\!\!\!
&\\&&\!\!\!\!K(\varepsilon)\!\!\!\!&&
\\&\!\!\!\!-K(\varepsilon)\!\!\!\!&&&
\\K(\varepsilon)\!\!\!\!&&&&0}},
\text{ in which }K(\varepsilon):=\begin{cases}
\mat{1&0\\0&1}\text{ if }\varepsilon =(-1)^{n+1},\\[12pt]
\mat{0&-1\\1&0}\text{ if }\varepsilon =(-1)^{n}.
\end{cases}
\]
If $\ff$ is (C) and $\alpha,\beta \in\pp$ , then we define the $2n\times 2n$ matrix
\begin{equation*}\label{aam}
J_n(\alpha+i\beta)^{\pp}:=\mat{M&I_2&&0\\
&M&\ddots\\&&\ddots&I_2\\0&&&M}\quad\text{with }M:=\mat{\alpha&-\beta\\ \beta&\alpha},
\end{equation*}
which is the realification of $J_n(\alpha+i\beta)$.

Canonical matrices of selfadjoint and skewadjoint operators on a vector space over $\cc$ or $\rr$
with  nondegenerate
symmetric or Hermitian form  are given in
\cite{dok1,goh_old,goh,lanc}, or see \cite[Theorems 7.2, 7.3, 8.2, 8.3]{mehl}  and \cite[Theorems 5.1, 5.2]{mehl1}. In Section \ref{vvz}, we reprove these results over (A)--(C) in a uniform manner and obtain the following corollary.

\begin{corollary}[of Theorem \ref{Theorem 6}]
\label{theor}
Let $\mc A:V\to V$ be a $\zeta $-adjoint
operator on a vector
space $V$  over a field\/ $\mathbb F$
with nondegenerate
$\varepsilon$-Hermitian form
$\mc F:V\times V\to\ff$, where
$\varepsilon,\zeta \in\{1,-1\}$. If $\mathbb F$ is one
of the fields {\rm(A)--(C)}, then there
exists a basis of\/ $V$ in which the
pair $(\mc A,\mc  F)_V$ is given by a direct
sum, uniquely determined up to
permutation of summands, of matrix pairs of the following types,
respectively\footnote{``Up to replacement by $\zeta \lambda $'' in (A) means that $\lambda $ and $\zeta \lambda $ give the same $(\mc A,\mc  F)_V$ in different bases; ``$(J_n(\lambda ),\pm Z_n)$'' in  (B) is the abbreviation of ``$(J_n(\lambda ),Z_n)$ and $(J_n(\lambda ),-Z_n)$''.}:
\begin{itemize}
 \item[\rm(A)]
For every $\lambda \in\ff$ determined up to replacement by $\zeta \lambda $,\\
 $\bullet$ if $\lambda \ne 0$ then\\
$\quad
\begin{cases}
(J_n(\lambda ),Z_n)&\text{if $\varepsilon=\zeta =1$},\\
\left(\mat{J_n(\lambda)&0\\0& \zeta
J_n(\lambda)^T},
\mat{
0&\varepsilon I_n\\I_n&0
}\right)&\text{if }(\varepsilon,\zeta)\ne(1,1);
\end{cases}
$
\smallskip

\noindent
$\bullet$ if $\lambda = 0$ then\\
$\quad \arraycolsep=2pt
\begin{cases}
(J_n(0),Z_n(\zeta ))&\!\!\!\!\text{if $\varepsilon =1$ for $n$ odd, $ \varepsilon=\zeta$ for $n$ even},
      \\
\left(\mat{J_n(0)&0\\0& \zeta
J_n(0)^T},
\mat{
0&\varepsilon I_n\\I_n&0
}\right)&\!\!\!\!\text{if $\varepsilon =-1$ for $n$
odd, $ \varepsilon\ne\zeta$ for $n$ even}.
\end{cases}
$

  \item[\rm(B)]
For $\varepsilon=\zeta= 1$ (see \eqref{cfm}) and for  every $\lambda \in\ff$ determined up to replacement by $\ba \lambda $,\\
$\quad
\begin{cases}
(J_n(\lambda ),\pm Z_n)&\text{if $\lambda \in\pp$},
      \\
\left(\mat{J_n(\lambda)&0\\0&
J_n(\lambda)^*},
\mat{
0& I_n\\I_n&0
}\right)&\text{if $\lambda \notin\pp$}.
\end{cases}
$

  \item[\rm(C)]\ {\rm(c$_1$)} For every $a \in\pp$ determined up to replacement by $\zeta a$,\\
 $\bullet$ if $a \ne 0$ then\\
$\quad
\begin{cases}
(J_n(a),\pm Z_n)&\text{if $\varepsilon=\zeta =1$},\\
\left(\mat{J_n(a)&0\\0& \zeta
J_n(a)^T},
\mat{
0&\varepsilon I_n\\I_n&0
}\right)&\text{if }(\varepsilon,\zeta)\ne(1,1);
\end{cases}
$
\smallskip

\noindent
$\bullet$ if $a = 0$ then\\
$\quad \arraycolsep=2pt
\begin{cases}
(J_n(0),\pm Z_n(\zeta ))&\!\!\!\!\text{if $\varepsilon =1$ for $n$ odd, $ \varepsilon=\zeta$ for $n$ even},
      \\
\left(\mat{J_n(0)&0\\0& \zeta
J_n(0)^T},
\mat{
0&\varepsilon I_n\\I_n&0
}\right)&\!\!\!\!\text{if $\varepsilon =-1$ for $n$
odd, $ \varepsilon\ne\zeta$ for $n$ even}.
\end{cases}
$
\smallskip

\ \ {\rm(c$_2$)}
For every $a\in \pp$ determined up to replacement by $\zeta a$ and every nonzero $b\in \pp$ determined up to replacement by $-b$,\\
$\bullet$ if $\zeta =1$ then\\
$\quad
\begin{cases}
(J_n(a+ib)^{\pp},Z_{2n})&\text{if $\varepsilon=1$},
                                                   \\
\left(\mat{J_n(a+ib)^{\pp}&0\\0& \left(J_n(a+ib)^{\pp}\right)^T},
\mat{
0&- I_{2n}\\I_{2n}&0
}\right)&\text{if $\varepsilon=-1$;}
\end{cases}
$
\smallskip

\noindent $\bullet$  if $\zeta =-1$ then\\
$\quad
\begin{cases}
(J_n(ib)^{\pp},\pm K_n(\varepsilon))&\text{if $a=0$},
                                                   \\
\left(\mat{J_n(a+ib)^{\pp}&0\\0& -\left(J_n(a+ib)^{\pp}\right)^T},
\mat{
0&\varepsilon  I_{2n}\\I_{2n}&0
}\right)&\text{if $a\ne 0$}.
\end{cases}
$
\end{itemize}
\end{corollary}

Analogous canonical matrices over (A)--(C) (and over the skew field of quaternions) of isometric operators on vector spaces with nondegenerate
symmetric/Hermitian form and of bilinear/sesquilinear forms  are given in a uniform manner in \cite[Theorem 2.1]{ser_iso} and \cite[Theorem 2.1]{hor-ser_can1}; they are easily derived from \cite[Theorems 5 and 3]{ser_izv}.

Note that the problems of classifying isometric operators and selfadjoint operators on a vector space with a possibly degenerate symmetric or Hermitian form are wild; see  \cite[Theorem 6.1]{ser_iso}. Recall that a classification problem is \emph{wild} if it contains the problem of classifying pairs of linear operators, and hence (see   \cite{bel}) the problem of classifying arbitrary systems of linear operators.

\section{Operators on a vector space with diagonalizable form}\label{jbt}

Let $\ff$ be a field of characteristic different from $2$. Let $\ff$ be a quadratic extension of its subfield $\mathbb K$ (i.e., $\dim_{\mathbb K}\ff=2$). By \cite[Chapter VI, \S 2, Example 1]{lan}, there exists $j\in\ff$ such that $\ff=\mathbb K(j)$ and $j^2\in\mathbb K$. Then $\ff=\mathbb K+j\mathbb K$ and $a+bj\mapsto a-bj$ ($a,b\in\mathbb K)$ is an involution on $\ff$.

Conversely, let $\ff$ be a field of characteristic different from 2 with nonidentity involution $a\mapsto \ti a$. Let us choose $a\in\ff$ such that $\ti a\ne a$ and write $j:=a-\ti a$. Then $\ti{\jmath}=-j$ and $\widetilde {j^2}=\ti{\jmath}\ti{\jmath}=j^2$.
The result of Artin \cite[Chapter VI, Theorem 1.8]{lan} ensures that $\ff$ has dimension $2$  over the fixed field  $\ff_0:=\{a\in\ff\,|\,\ti a=a\}$, and so
\begin{equation}\label{uht}
\ff=\ff_0+j\ff_0,\quad \ti{\jmath}=-j,\quad j^2\in\ff_0.
\end{equation}
We do not
consider skew-Hermitian forms over $\ff$ with nonidentity involution since if $\mc F$ is skew-Hermitian over $\ff$, then
\begin{equation}\label{jna}
 j\mc F\ \text{ is a Hermitian form}.
\end{equation}

Let $\ff^{\times}$ and $\ff^{\times}_0$ be the multiplicative groups of $\ff$ and  $\ff_0$.
We choose an element in each coset of the quotient group
$
\ff^{\times}/\ff^{\times}_0
$
and denote by $r(\ffff)$ the set of chosen elements.
In particular, if $\ff=\cc$, then $\ff_0=
\rr$ and we can take
\begin{equation}\label{rrd}
r(\ffff)=
r(\cc^{\times}/\rr^{\times})=\{\cos \varphi +i\sin \varphi\,|\,0\le \varphi<\pi\}.
\end{equation}

\begin{lemma}\label{iuj1}
{\rm(a)} Let $\ff$ be a field of characteristic different from 2 with nonidentity involution.
Let $V$ be a vector space over $\ff$ with nondegenerate diagonalizable sesquilinear form $\mc F:V\times V\to\ff$.
Let
$\mc A:V\to V$ be an isometric, selfadjoint, or skewadjoint linear operator.
Then there exists a basis of\/ $V$, in which $\mc F$ is given by a diagonal matrix
\begin{equation}\label{njh1}
D= e_1D_1\oplus\dots\oplus e_tD_t,\qquad e_1,\dots,e_t\in r(\ffff),
\end{equation}
in which  all $e_1,\dots,e_t$ are distinct and $D_1,\dots,D_t$ are diagonal matrices over $\ff_0$.
In each such basis, the matrix of $\mc A$ has the block diagonal form
\begin{equation*}\label{rtf2}
A= A_1\oplus\dots\oplus A_t,\qquad\text{every $A_i$ has the size of }D_i.
\end{equation*}

{\rm(b)} If $\ff=\cc$, then the decomposition \eqref{njh1} can be taken in the form
\begin{equation}\label{njh4}
D= e_1\mat{I_{p_1}&0\\0&-I_{q_1}}\oplus\dots\oplus e_t\mat{I_{p_t}&0\\0&-I_{q_t}},
\end{equation}
in which
$
e_1,\dots,e_t\in\{\cos \varphi +i\sin \varphi\,|\,0\le \varphi<\pi\}
$  are distinct, $p_l,q_l\in\{0,1,2,\dots\}$ and $p_l+q_l\ge 1$ for all $l=1,\dots,t$.
\end{lemma}

\begin{proof} (a)
Let the form $\mc F$ be given by a nonsingular diagonal matrix $D=\diag(d_1,\dots,d_n)$ in some basis of $V$. Permuting the diagonal entries, we obtain $D$ of the form \eqref{njh1} in some basis of $V$. Then
\begin{equation}\label{hkt}
e_t^{-1}D=D'\oplus D_t,\qquad D':=e_t^{-1}(e_1D_1\oplus\dots\oplus e_{t-1}D_{t-1} ).
\end{equation}
Write
\begin{equation*}\label{njh1a}
D'=\Delta_0+j\Delta_1,\qquad
\text{$\Delta_0$ and $\Delta_1$ are over $\ff_0$.}
\end{equation*}
For each index $i<t$, we have $e_i\ff_0\ne e_t\ff_0$, and so $e_t^{-1}e_i\notin\ff_0$. Hence
\begin{equation}
\label{uuh}
\text{the matrix $\Delta_1$ is nonsingular.}
\end{equation}
Since $D'$ and $D_t$ are diagonal matrices,  \begin{equation}\label{ddm}
D'-D^{\prime {\as}}=D'-\tilde{ D'}=(\Delta_0+j\Delta_1)
-(\Delta_0-j\Delta_1)=2j\Delta_1,\quad D_t-D_t^{\as}=0.
\end{equation}

Let $\mc A:V\to V$ be an operator and let $A$ be its matrix in the basis, in which  $\mc F$ is given by the matrix
\eqref{njh1}. We partition it into blocks conformally with \eqref{hkt}:
\begin{equation}\label{nbo}
A=\mat{B&X
\\Y&A_t},\qquad A_t\text{ has the size of }D_t.
\end{equation}

\emph{Case 1: $\mc A$ is an isometric operator.} The equality $\mc F(\mc A u,\mc A v)=\mc F(u, v)$ implies that $A^{\as}DA=D$. Thus, $A^{\as}(e_t^{-1}D)A=e_t^{-1}D$; i.e,
\begin{equation}\label{hte}
\mat{B^{\as}&Y^{\as}
\\X^{\as}&A_t^{\as}}\mat{D'&0\\0&D_t}\mat{B&X
\\Y&A_t}=\mat{D'&0\\0&D_t}.
\end{equation}
It suffices to prove that $X=0$ and $Y=0$ since then $B^{\as}D'B= D'$  and we can use induction on $t$.

By \eqref{hte},
\begin{align}\label{a1}
B^{\as}D'B+Y^{\as}D_tY&=D',\\\label{a2}
B^{\as}D'X+Y^{\as}D_tA_t&=0,\\\label{a3}
X^{\as}D'B+A_t^{\as}D_tY&=0,
\\\label{a4}
X^{\as}D'X+A_t^{\as}D_tA_t&=D_t.
\end{align}

Applying the star to \eqref{a1} and subtracting the obtained equality from \eqref{a1}, we get
\begin{equation*}\label{dfg1}
B^{\as}(D'-D^{\prime {\as}})B+Y^{\as}(D_t-D_t^{\as})Y=D'-D^{\prime {\as}}.
\end{equation*}
By \eqref{uuh} and  \eqref{ddm},
$B^{\as}\Delta_1B=\Delta_1$ and the matrix $B$ is nonsingular.

Applying the star to \eqref{a3} and subtracting the obtained equality from \eqref{a2}, we get
\begin{equation*}\label{dfg}
B^{\as}(D'-D^{\prime {\as}})X+Y^{\as}(D_t-D_t^{\as})A_t=0.
\end{equation*}
By \eqref{ddm}, $B^{\as}\Delta_1X=0.$
Since $B$ and $\Delta_1$ are nonsingular, $X=0$.
The equality \eqref{a4} implies that $A_t^{\as}D_tA_t=D_t$. Hence $A_t$ is nonsingular. By \eqref{a3}, $A_t^{\as}D_tY=0$, and so $Y=0$.

Therefore, $A=B\oplus A_t$, which proves the lemma if the operator $\mc A$ is isometric.

\emph{Case 2: $\mc A$ is a $\zeta $-adjoint operator; $\zeta \in\{1,-1\}$.}
The equality $\mc F(u,\mc A v)=\zeta  \mc F(\mc A u, v)$ implies that its matrix \eqref{nbo} satisfies $DA=\zeta  A^{\as}D$. With the notation \eqref{hkt}, we have
\begin{equation}\label{htej}
\mat{D'&0\\0&D_t}\mat{B&X
\\Y&A_t}=\zeta  \mat{B^{\as}&Y^{\as}
\\X^{\as}&A_t^{\as}}\mat{D'&0\\0&D_t}.
\end{equation}
It suffices to prove that $X=0$ and $Y=0$ since then $D'B= \zeta  B^{\as}D'$ and we can use induction on $t$.

By \eqref{htej},
\begin{align}\label{b2}
D'X&=\zeta  Y^{\as}D_t,\\\label{b3}
D_tY&=\zeta  X^{\as}D'.
\end{align}

We apply the star to \eqref{b3}, multiply it by $\zeta $, interchange its sizes, substitute the obtained equality from \eqref{b2}, and get
\begin{equation*}\label{dfgd}
(D'-D^{\prime {\as}})X=\zeta  Y^{\as}(D_t-D_t^{\as}).
\end{equation*}
By \eqref{ddm}, $\Delta_1X=0$. By \eqref{uuh}, $X=0$. The equality \eqref{b3} ensures that $Y=0$.

(b) Let $\ff=\cc$. We take $r(\ffff)$ in the form
\eqref{rrd}.
 Let the form $\mc F$ be given by a nonsingular diagonal matrix $D=\diag(d_1,\dots,d_n)$. For each index $i$, there exist $a_i\in\rr$ and $c_i\in \cc$ with $|c_i|=1$ such that $d_i=c_ia_i^2$. Replacing $D$ by
$S^*DS$ with $S:=\diag(1/a_1,\dots,1/a_n)$
 and permuting the diagonal entries, we obtain $D$ of the form \eqref{njh4}.
\end{proof}

Two pairs $(\mc A,\mc F)_V$ and $(\mc B,\mc H)_W$  consisting of linear mappings and sesquilinear forms on vector spaces $V$ and $W$ are \emph{isomorphic} if there exists a linear bijection $\varphi :V\to W$ that transforms $\mc A$ to $\mc B$ and $\mc F$ to $\mc H$; that is,
$\varphi \mc A=\mc B\varphi $ and $\mc F(v,v')=
\mc H(\varphi v,\varphi v')$ for all $v,v'\in V$.

The pairs $(\mc A_i,\mc F_i)_{V_i}$ from the following theorem are classified in \cite[Theorems 5 and 6]{ser_izv}  and in Theorem \ref{Theorem 6}  up to classification of symmetric and Hermitian forms over finite extensions of $\ff$.

\begin{theorem}\label{yyj}
Let $\ff$ be a field of characteristic different from 2 with nonidentity involution.
Let $V$ be a vector space over $\ff$ with nondegenerate diagonalizable sesquilinear form $\mc F:V\times V\to\ff$.
Let
$\mc A:V\to V$ be an isometric (respectively, $\zeta $-adjoint with $\zeta \in\{1,-1\}$) linear operator.
Then
\begin{equation}\label{rmr}
(\mc A,\mc F)_V=(\mc A_1,e_1\mc F_1)_{V_1}\oplus\dots\oplus (\mc A_t,e_t\mc F_t)_{V_t},\quad e_1,\dots,e_t\in r(\ffff),
\end{equation}
in which all $e_1,\dots,e_t$ are distinct and each $\mc A_i$ is an isometric (respectively, $\zeta  $-adjoint) linear operator on a vector space $V_i$ over $\ff$ with a nondegenerate Hermitian form $\mc F_i:V_i\times V_i\to\ff$.
The summands in \eqref{rmr} are uniquely determined by $(\mc A,\mc F)_V$, up to permutations and replacements of all $(\mc A_i,\mc F_i)_{V_i}$ by isomorphic pairs.
\end{theorem}

\begin{proof}
By Lemma \ref{iuj1}, there exists a decomposition $V=V_1\oplus\dots\oplus V_t$ that ensures
\eqref{rmr}. It remains to prove that the summands in \eqref{rmr} are uniquely determined, up to  permutations and isomorphisms.

If $\ff=\cc$, then the uniqueness follows from the fact that each system of linear mappings and sesquilinear forms over $\cc$ is uniquely decomposed, up to  permutations and isomorphisms of summands, into a direct sum of indecomposable systems; see \cite[Theorem  2]{ser_izv} or \cite[Corollary of Theorem 3.2]{ser_iso}.

If $\ff\ne\cc$, the proof is more complicated.
We use the method that was developed in \cite{ser_izv} (or see \cite{ser_iso}). Systems of linear mappings and sesquilinear forms are considered as representations of \emph{mixed graphs}; i.e., graphs with undirected and directed edges. Directed edges represent linear mappings, and undirected edges represent forms.
Each pair $(\mc A,\mc F)_V$ from Theorem \ref{yyj} defines the representation
\begin{equation}\label{ffz}
\mc P:\xymatrix@R=6pt{V
\save !<5pt,0cm>
\ar@{-}@(ur,dr)^{\mc F}
\restore
\save !<-5pt,0cm>
\ar@{->}@(ul,dl)_{\mc A}
\restore
}\quad\text{of
the mixed graph}\quad G:  \xymatrix@R=6pt{v
\ar@{-}@(ur,dr)^{\gamma }
\ar@{->}@(ul,dl)_{\alpha }
}
\end{equation}
Let $V^{\as}$ be the \emph{$^{\as}$dual space} of all semilinear forms $\varphi :V\to\ff$. The linear mapping $\mc A:V\to V$ defines the \emph{$^{\as}$adjoint linear mapping} $\mc A^{\as}:V^{\as}\to V^{\as}$ defined by
$(\mc A^{\as}\varphi )v:=\varphi (\mc A v)$ for all $v\in V$ and $\varphi \in V^{\as}$. The sesquilinear form
$\mc F:V\times V\to\ff$ defines both the linear mapping (we denote it by the same letter)
$\mc F: V\to V^{\as}$ via $v\mapsto \mc F(?,v)$
and the $^{\as}$adjoint linear mapping
$\mc F^{\as}:V\to V^{\as}$ via $u\mapsto \widetilde{\mc F(v,?)}.$
The representation $\mc P$ in \eqref{ffz} defines the representation
\begin{equation*}\label{bjr}
\underline{\mc P}:\xymatrix@R=6pt{V
\save !<-5pt,0cm>
\ar@{->}@(ul,dl)_{\mc A}
\restore
\ar@/^0.5pc/@{->}[r]^{\mc F}
\ar@/_0.5pc/@{->}[r]_{\mc F^{\as}}
                         &
                         V^{\as}
\save !<7pt,0cm>
\ar@{->}@(ur,dr)^{\mc A^{\as}}
\restore
}\text{\quad  of the quiver \ }
\underline{G}:\xymatrix@R=6pt{v
\save !<-2pt,0cm>
\ar@{->}@(ul,dl)_{\alpha}
\restore
\ar@/^0.5pc/@{->}[r]^{\gamma}
\ar@/_0.5pc/@{->}[r]_{\gamma^*}
                         &
                         v^*
\save !<4pt,0cm>
\ar@{->}@(ur,dr)^{\alpha^*}
\restore
}
\end{equation*}

Theorem 1 from \cite[]{ser_izv} (or see \cite[Theorem 3.1]{ser_iso}) ensures the following statement:
\begin{equation}
\label{pot}
\parbox[c]{0.84\textwidth}{Let $\mc P=\mc P_1\oplus\dots\oplus\mc P_t=
\mc R_1\oplus\dots\oplus\mc R_t$. Let all indecomposable direct summands of $\un{\mc P_i}$ be not isomorphic to all indecomposable direct summands of $\un{\mc R_j}$ if $i\ne j$.
Then ${\mc P_i}$ is isomorphic to ${\mc R_i}$ for all $i$.
}
\end{equation}

The sesquilinear form $\mc F:V\times V\to \ff$  from the pair \eqref{rmr} is given by a  nonsingular diagonal matrix $F=\diag(\lambda_1,\dots,\lambda_n)$  in some basis of $V$.
Let $H=\diag(\mu_1,\dots,\mu_n)$ be its matrix in another basis. Then $S^{\as}FS=H$ for some nonsingular $S\in\ff^{n\times n}$. Hence $F^{-\as}F$ and
$H^{-\as}H$ are similar via $S$, and so there is a renumbering of $\mu_1,\dots,\mu_n$ such that
$\ti{\lambda}_i^{-1}\lambda_i =
\ti{\mu}_i^{-1}\mu_i$ for $i=1,\dots,n$. We have
$\lambda_i\ba{\mu}_i =\ba{\lambda}_i
\mu_i\in\ff_0$ and
\begin{equation*}\label{vgf}
 \lambda_i\ff_0=  \lambda_i(\lambda_i\ba{\mu}_i)^{-1}\ff_0
 =\ba{\mu}_i^{-1}\ff_0=
\ba{\mu}_i^{-1}\ba{\mu}_i{\mu}_i\ff_0
={\mu}_i \ff_0.
\end{equation*}

Therefore, $t$, $e_1,\dots,e_t$, and $\dim V_1,\dots,\dim V_t$ in \eqref{rmr} are uniquely determined by $\mc F$. Let us consider the decomposition \eqref{rmr} and another decomposition
\begin{equation*}\label{rmj}
(\mc A,\mc F)_V=(\mc B_1,e_1\mc H_1)_{W_1}\oplus\dots\oplus (\mc B_t,e_t\mc H_t)_{W_t},
\end{equation*}
in which the forms $\mc H_1,\dots,\mc H_t$ are Hermitian.
Let
\begin{equation}\label{ssd}
\mc P=\mc P_1\oplus\dots\oplus\mc P_t\quad
\text{and}\quad\mc P=\mc R_1\oplus\dots\oplus\mc R_t
\end{equation}
be
the corresponding decompositions of the representation $\mc P$ of the mixed graph $G$.
Then $\mc P_i:\xymatrix@R=6pt{V_i
\save !<5pt,0cm>
\ar@{-}@(ur,dr)^{e_i\mc F_i}
\restore
\save !<-5pt,0cm>
\ar@{->}@(ul,dl)_{\mc A_i}
\restore
}$
and
$\mc R_j:\xymatrix@R=6pt{W_j
\save !<5pt,0cm>
\ar@{-}@(ur,dr)^{e_i\mc H_j}
\restore
\save !<-5pt,0cm>
\ar@{->}@(ul,dl)_{\mc B_j}
\restore
}$
The corresponding representations
\[
\un{\mc P_i}:\xymatrix@R=6pt{V_i
\save !<-2pt,0cm>
\ar@{->}@(ul,dl)_{\mc A_i}
\restore
\ar@/^0.5pc/@{->}[rr]^{e_i\mc F_i}
\ar@/_0.5pc/@{->}[rr]_{\ba{e}_i\mc F_i}
                         &&
                         V_i^{\as}
\save !<4pt,0cm>
\ar@{->}@(ur,dr)^{\mc A_i^{\as}}
\restore
}
\qquad
\un{\mc R_j}:\xymatrix@R=6pt{W_j
\save !<-5pt,0cm>
\ar@{->}@(ul,dl)_{\mc B_j}
\restore
\ar@/^0.5pc/@{->}[rr]^{e_j\mc H_j}
\ar@/_0.5pc/@{->}[rr]_{\ba{e}_j\mc H_j}
                         &&
                         W_j^{\as}
\save !<6pt,0cm>
\ar@{->}@(ur,dr)^{\mc B_j^{\as}}
\restore
}
\]
of the quiver $\un G$ are isomorphic to
\begin{equation}\label{kjk}
\xymatrix@=39pt{V_i
\save !<-2pt,0cm>
\ar@{->}@(ul,dl)_{\mc A_i}
\restore
\ar@/^0.5pc/@{->}[r]^{\ba{e}_i^{-1}e_i1_{V_i}}
\ar@/_0.5pc/@{->}[r]_{1_{V_i}}
                         &
                         V_i
\save !<2pt,0cm>
\ar@{->}@(ur,dr)^{\mc F_i^{-1}\mc A_i^{\as}\mc F_i}
\restore
}
\quad
\xymatrix@=39pt{W_j
\save !<-5pt,0cm>
\ar@{->}@(ul,dl)_{\mc B_j}
\restore
\ar@/^0.5pc/@{->}[r]^{\ba{e}_j^{-1}e_j1_{W_j}}
\ar@/_0.5pc/@{->}[r]_{1_{W_j}}
                         &
                         W_i
\save !<5pt,0cm>
\ar@{->}@(ur,dr)^{\mc H_j^{-1}\mc B_j^{\as}\mc H_j}
\restore
}
\end{equation}

Let $i\ne j$. Then $\ba{e}_i^{-1}e_i\ne \ba{e}_j^{-1}e_j$. By Jordan canonical form under similarity, the representations \eqref{kjk} have no common indecomposable summands. By \eqref{pot}, the summands $\mc P_i$ and $\mc R_i$ in \eqref{ssd} are isomorphic for all $i$.
Hence,
the summands in \eqref{rmr} are uniquely determined up to permutations and isomorphisms.
\end{proof}

Let us apply Lemma \ref{iuj1} to the group
\[
\mc U(\mc F):=\{\mc A:V\to V\,|\,  \mc F(\mc A u,\mc A v)=\mc F(u, v)\text{ for all }u,v\in V \}
\]
of isometric operators on the complex vector space $V$ with  nondegenerate diagonalizable sesquilinear form $\mc F$. The group $\mc U(\mc F)$ is the group of isometries of  $\mc F$. (\DJ okovi\'{c} \cite{dok} and Szechtman \cite{sz} study the structure of the group of isometries of an arbitrary bilinear form.)
The matrices of $\mc U(\mc F)$ form the matrix group
\[
U(D):=\{A\in\cc^{n\times n}\,|\,A^*DA=D\},
\]
in which $D$ is the matrix \eqref{njh4} of $\mc F$.
Its special case is the \emph{indefinite unitary group} (which is also called the \emph{pseudo-unitary group})
\[
U(p,q):=\{A\in\cc^{(p+q)\times (p+q)}\,|\,A^*I_{p,q}A=I_{p,q}\},
\quad
I_{p,q}:=\mat{I_{p}&0\\0&-I_{q}}.
\]

For each  $n\times n$ complex matrix $E$, we define the Lie matrix algebra
\[
S(E):=\{A\in\cc^{n\times n}\,|\,EA=  -A^*E\}.
\]

\begin{corollary}[of Lemma \ref{iuj1}]\label{uyn}
{\rm(a)}  The group of isometric operators on a complex vector space with nondegenerate diagonalizable sesquilinear form is isomorphic to a direct product of indefinite unitary groups. If $D$ is the diagonal matrix \eqref{njh1}, then
\begin{equation*}\label{drj}
U(D)=U(p_1,q_1)\times\dots\times U(p_r,q_r).
\end{equation*}

{\rm(b)}  The Lie algebra of skewadjoint  operators on a complex vector space with nondegenerate diagonalizable sesquilinear form is isomorphic to a direct product of Lie
algebras of skewadjoint  operators on indefinite inner product spaces. If $D$ is the diagonal matrix \eqref{njh1}, then
\begin{equation*}\label{drj1}
S(D)=
S(I_{p_1,q_1})\times\dots\times S(I_{p_r,q_r}).
\end{equation*}

\vskip-2em
 \hfill \hbox{\qedsymbol}
\end{corollary}

\section{Classification of skewadjoint operators}
\label{s_self}

Every square matrix over
$\ff$ is similar to a
direct sum, determined uniquely up to permutation of summands, of {\it Frobenius
blocks}
\[
\begin{bmatrix} 0&&0
&-c_0\\1&\ddots&&-c_1
\\&\ddots&0&\vdots\\
0&&1& -c_{n-1} \end{bmatrix}
\]
whose characteristic
polynomials
$\chi(x) =c_0+c_1x+\dots+c_{n-1}x^{n-1}+x^n
$
are integer powers of irreducible
polynomials.

Let ${\cal O}_{\mathbb F}$ be obtained from the set of Frobenius blocks over $\ff$ by replacing each Frobenius block by a similar matrix. Thus, the characteristic
polynomial $\chi_{\Phi}(x)$ of $\Phi\in {\cal O}_{\mathbb F}$ is an integer power of an irreducible
polynomial $p_{\Phi}(x)$. For example, ${\cal O}_{\mathbb F}$ can consist of all Frobenius blocks if $\ff$ is an arbitrary field; ${\cal O}_{\mathbb F}$ can consist of all Jordan blocks if $\ff=\cc$.

For each
$\Phi\in \mc O_{\ff}$ and $\varepsilon,\zeta\in\{1,-1\}$,  if there
exists a nonsingular matrix
$M$ satisfying
\begin{equation}\label{jffc}
M=\varepsilon  M^{\as}, \qquad
M\Phi=\varepsilon\zeta  (M\Phi)^{\as},
\end{equation}
 then
we choose one and denote
it by $\Phi_{\varepsilon\zeta}$.
The existence conditions and
explicit form of $\Phi_{\varepsilon\zeta}$ for Frobenius blocks $\Phi$ are given in Lemma \ref{THEOREM 8}.

For each polynomial $f(x)=a_0+a_1x+\dots+a_{n}x^{n}\in\ff[x]$, we write
\begin{equation*}\label{ook}
\ba f(x)=\ba a_0+\ba a_1x+\dots+\ba a_{n}x^{n}.
\end{equation*}

Selfadjoint operators on  a vector space $V$ with  a nondegenerate symmetric or Hermitian form over a field $\ff$ of characteristic different from 2 are classified in \cite[Theorem 6]{ser_izv} up  to classification of symmetric and Hermitian forms over finite extensions of $\ff$.
The following theorem is an extended version of \cite[Theorem 6]{ser_izv}, which is supplemented by skewadjoint operators. It is formulated as \cite[Theorem 6]{ser_izv}.

\begin{theorem}\label{Theorem 6}
Let $\mc A$ be a $\zeta $-adjoint
operator on a vector
space $V$ with nondegenerate
$\varepsilon$-Hermitian form
$\mc F$ over a field $\ff$
of characteristic different from $2$ with
involution, where $\varepsilon,\zeta \in\{1,-1\}$ and
$\varepsilon= 1$ for
nonidentity involution on
$\ff$ (see \eqref{jna}). Then there
exists a basis of\/ $V$ in which the
pair $(\mc A,\mc  F)_V$ is given by a
direct sum of matrix pairs of
the following types:
\begin{itemize}
  \item[\rm(i)]
$A^{f(x)}_{\Phi}:=(\Phi,
\Phi_{\varepsilon\zeta}f(\Phi))$, in which $\Phi\in \mc O_{\ff}$
is a matrix for which
$\Phi_{\varepsilon\zeta}$ exists, $0\ne
f(x)=\ba f(\zeta x)\in \ff[x]$, and
$\deg(f(x))<\deg(p_{\Phi}(x))$.

  \item[\rm(ii)]
$\left(\mat{\Phi&0\\ 0&\zeta \Phi^{*}},\mat{
0&\varepsilon I_n\\I_n&0}\right)$, in which $\Phi\in \mc O_{\ff}$
is a matrix for which
$\Phi_{\varepsilon\zeta}$ does not
exist.

\end{itemize}

The summands are determined
to the following extent:
\begin{description}
  \item [Type (i)]
up to replacement of the
whole group of summands
$
A_{\Phi}^{f_1(x)}
\oplus\dots\oplus
A_{\Phi}^{f_s(x)}
$
with the same $\Phi$ by
$
A_{\Phi}^{g_1(x)}
\oplus\dots\oplus
  A_{\Phi}^{g_s(x)}
$
such that the Hermitian forms
\begin{equation}\label{bhy}
\begin{split}
&f_1(\omega)x_1^{\circ}x_1+\dots+
f_s(\omega)x_s^{\circ}x_s,
\\
&g_1(\omega)x_1^{\circ}x_1+\dots+
g_s(\omega)x_s^{\circ}x_s
\end{split}
\end{equation}
are equivalent over the
field
\begin{equation}\label{aza}
\ff[\omega]=\ff[x]/p_{\Phi}(x)\ff[x]\text{ with involution } f(\omega)\mapsto
f(\omega)^{\circ}= \ba
f(\zeta \omega).
\end{equation}

  \item [Type (ii)]
up to replacement of $\Phi$
by $\Psi\in \mc O_{\ff}$ with
$\chi_{\Psi}(x)
=\zeta^{\deg(\chi_{\Psi})}\ba\chi_{\Phi}(\zeta x)$.
\end{description}
\end{theorem}

\begin{proof}
This proof is a slight modification of the proof of \cite[Theorem 6]{ser_izv}, which uses the method of reducing the problem of classifying systems of forms and linear mappings to the problem of classifying systems of linear mappings. This method is developed in \cite{ser_izv} and is applied to the problems of classifying bilinear/sesquilinear forms, pairs of symmetric/skewsymmetric/Hermitian forms, and isometric/selfadjoint operators on a vector space with nondegenerate symmetric/skewsymmetric/Hermitian form in \cite[Theorems 3--6]{ser_izv}.
It is presented in detail in \cite{ser_iso} (see also \cite{hor-ser_can1,hor-ser_mixed,mel,ser_brazil}).

The pair $(\mc A, \mc F)_V$ satisfies the conditions
\[
\mc F(u,\mc Av)=\zeta   \mc F(\mc Av,u)
,\qquad\mc F(u,v)=\varepsilon \widetilde{\mc F(v,u)},\qquad
 \det \mc F\ne 0;
\]
therefore, it
defines the representation $\xymatrix@R=6pt{V
\save !<3pt,0cm>
\ar@{-}@(ur,dr)^{\mc F}
\restore
\save !<-4pt,0cm>
\ar@{->}@(ul,dl)_{\mc A}
\restore
}$ of the mixed graph with relations
\begin{equation}\label{rhn}
G:  \xymatrix@R=6pt{v
\ar@{-}@(ur,dr)^{\gamma }
\ar@{->}@(ul,dl)_{\alpha }
}  \qquad
\gamma \alpha=\zeta \alpha^*\gamma ,\ \
\gamma=\varepsilon \gamma^*,\ \
 \det \gamma\ne 0.
\end{equation}
Its quiver with involutions is
\begin{equation}\label{bhd2}
\underline{G}: \xymatrix@R=6pt{v
\save !<-2pt,0cm>
\ar@{->}@(ul,dl)_{\alpha}
\restore
\ar@/^0.5pc/@{->}[rr]^{\gamma}
\ar@/_0.5pc/@{->}[rr]_{\gamma^*}
                         &&
                         v^*
\save !<4pt,0cm>
\ar@{->}@(ur,dr)^{\alpha^*}
\restore
}  \qquad
\gamma \alpha=\zeta \alpha^*\gamma ,\ \
\gamma=\varepsilon \gamma^*,\ \
 \det \gamma\ne 0.
\end{equation}

$1^{\circ}$ Let us describe the set $\ind(\un{G})$ of nonisomorphic indecomposable matrix representations
of $\un{G}$. Each matrix representation of
\eqref{bhd2} is of the form
\begin{equation*}\label{vdw}
\xymatrix@R=6pt{\ff^n
\save !<-7pt,0cm>
\ar@{->}@(ul,dl)_{A_{\alpha}}
\restore
\ar@/^0.5pc/@{->}[rr]^{A_{\gamma}}
\ar@/_0.5pc/@{->}[rr]_{A_{\gamma^*}=\varepsilon A_{\gamma}}
                         &&
                         \ff^n
\save !<7pt,0cm>
\ar@{->}@(ur,dr)^{A_{\alpha^*}}
\restore
}  \quad
A_{\gamma }A_{\alpha} =\zeta
A_{\alpha^*}A_{\gamma} ,\ \
 \det A_{\gamma}\ne 0;
\end{equation*}
we give it by the triple
$(A_{\alpha}, A_{\gamma},
A_{\alpha^*})$ of $n\times n$ matrices, in which $A_{\gamma}$ is
nonsingular and
$A_{\gamma }A_{\alpha} =\zeta
A_{\alpha^*}A_{\gamma}.$
The adjoint representation is
given by $$(A,B,C)^{\circ} =
(C^{\as},\varepsilon B^{\as},A^{\as}).$$

Every matrix representation of the
quiver $\un{G}$ is isomorphic to one
of the form $(A, I, \zeta A)$. The
set $\ind(\un{G})$ consists of the matrix
representations $(\Phi, I,
\zeta \Phi)$ in which $\Phi\in \mc O_{\ff}$.

$2^{\circ}$. Let us find the sets
$\ind_0(\un{G})$
and $\ind_1(\un{G})$. We have that
\[
(\Psi,I,\zeta \Psi)\simeq
(\Phi,I,\zeta \Phi)^{\circ} =
(\zeta \Phi^{\as},I,\Phi^{\as})
\]
if and only if $\Psi$ is
similar to $\zeta \Phi^{\as}$,
if and only if $\det (xI-\Psi)=\det (xI-\zeta \Phi^{\as})=
\det (xI-\zeta \ba{\Phi})=
\pm\det (\zeta xI- \ba{\Phi})$, if and only if
$\chi_{\Psi}(x)
=\pm\ba\chi_{\Phi}(\zeta x)$.

Suppose that $(\Phi,I,\zeta \Phi)$ is
isomorphic to a selfadjoint matrix
representation. By \cite[Lemma
6]{ser_izv}, there exists
an isomorphism
\[
 h=[I, H]:
 (\Phi,I,\zeta \Phi)\to (A,B,A^{\as}),
 \qquad B=\varepsilon B^{\as}.
\]
Then
\[
A = \Phi,\quad B = H,\quad B^{\as}
=\varepsilon H,\quad
A^{\as}H =\zeta H\Phi;
\]
i.e.,
\[
B =\varepsilon  B^{\as},
 \qquad
B\Phi =\zeta \Phi^{\as} B=
\varepsilon \zeta (B\Phi)^{\as}.\]
By \cite[Lemma 8]{ser_izv},  if $\Phi$ is nonsingular, then  $\varepsilon =1$ or $\varepsilon \zeta=1$ (i.e.,  $\varepsilon =1$ or $\varepsilon =\zeta =-1$); if $\Phi=J_n(0)$, then
$\varepsilon = 1$  for $n$ odd
and $\varepsilon \zeta = 1$ for $n$
even.

We can
take $B=\Phi_{\varepsilon \zeta}$.
Thus, the set
$\ind_0(\un{G})$
consists of the matrix
representations $A_{\Phi}: =
(\Phi,\Phi_{\varepsilon \zeta},\Phi^{\as}),$ in which
$\Phi\in \mc O_{\ff}$ is a matrix for which
$\Phi_{\varepsilon \zeta}$ exists. The set
$\ind_1(\un{G})$
consists of the matrix
representations
$(\Phi,I,\zeta \Phi)$, in which
$\Phi\in \mc O_{\ff}$ is a matrix for which
$\Phi_{\varepsilon \zeta}$ does not exist, and $\chi_{\Phi}(x)$ is
determined up to replacement
by $\pm \ba\chi_{\Phi}(\zeta x)$.

$3^{\circ}$. Let us describe the
orbits of representations
from $\ind_0(\un{G})$. Let $g=[G_1, G_2]\in
\End(A_{\Phi}).$ Then
\[
\Phi G_1 = G_1\Phi,\quad
\Phi_{\varepsilon \zeta} G_1 = G_2\Phi_{\varepsilon \zeta},\quad
\Phi^{\as} G_2 = G_2\Phi^{\as}.
\]
Since $G_1$ commutes with $\Phi$, which is similar to a
Frobenius block, we have
\begin{align*}
G_1 &= f(\Phi)\quad \text{for some }f(x)\in
 \ff[x],
 \\ G_2 &=
\Phi_{\varepsilon \zeta} f(\Phi) \Phi_{\varepsilon \zeta}^{-1}=
f(\Phi_{\varepsilon \zeta} \Phi \Phi_{\varepsilon \zeta}^{-1})=
f(\zeta\Phi^{\as}).
\end{align*}

Consequently,
$\End(A_{\Phi})=\{
[f(\Phi),f(\zeta\Phi^{\as})]\,|\,
f(x)\in \ff[x]\}$ is the algebra
with involution
\[
[f(\Phi),f(\zeta\Phi^{\as})]^{\circ}=
[\ba f(\zeta\Phi),f(\Phi)^{\as}].
\]
The field $T(A_{\Phi})$ can
be identified with the field
$ \ff[\omega] =  \ff[x]/p_{\Phi}(x) \ff[x]$ with involution
$f(\omega)^{\circ} = \ba
f(\zeta\omega)$. The orbit of $A_{\Phi}$ consists of the matrix
representations
$A_{\Phi}^{f(\omega)}:\xymatrix@R=6pt{\bullet
\save !<3pt,0cm>
\ar@{-}@(ur,dr)^{\Phi_{\varepsilon\zeta}f(\Phi)}
\restore
\save !<-4pt,0cm>
\ar@{->}@(ul,dl)_{\Phi}
\restore
}$
 of
\eqref{rhn} given by nonzero polynomials $f(x)\in \ff[x]$ of degree $<\deg(p_{\Phi}(x))$ that satisfy
$f(x)=\ba f(\zeta x)$.

$4^{\circ}$. From
$2^{\circ}$,  $3^{\circ}$,
and \cite[Theorem
1]{ser_izv},
the proof of Theorem
\ref{Theorem 6} now follows.
\end{proof}

Note that there is an unexpected bijective correspondence between all pairs $(\mc A,\mc  F)_V$ from Theorem \ref{Theorem 6} and all pairs
consisting of an $\varepsilon \zeta $-Hermitian form and a nondegenerate $\zeta $-Hermitian form on $V$. The pair corresponding to  $(\mc A,\mc  F)_V$ is  $(\mc E,\mc  F)_V$, in which  $\mc E:V\times V\to\ff$ is defined by
\[
\mc E(u,v):=\mc F(u,\mc Av)\text{\qquad for all } u,v\in V.
\]
This form is $\varepsilon \zeta $-Hermitian since
\[
\mc E(u,v)=\mc F(u,\mc Av)=\zeta\mc F(\mc Au,v)=
\varepsilon\zeta\widetilde{\mc F(v,\mc Au)}=
\varepsilon\zeta\widetilde{\mc E(v,u)}.
\]
Therefore, Theorem \ref{Theorem 6} could be derived from \cite[Theorem 4]{ser_izv}, in which the pairs of Hermitian, or symmetric, or skew-symmetric forms over a field of characteristic different from 2 are classified. However, the direct proof of Theorem \ref{Theorem 6} helps to construct simpler canonical forms.

The existence conditions and
explicit form of $\Phi_{\varepsilon\zeta}$ for Frobenius blocks $\Phi$ are given in the following lemma, which is an extended version of \cite[Theorem 8]{ser_izv}.
By \eqref{jffc}, $\Phi_{\varepsilon\zeta}$ is defined by the equalities
\begin{equation}\label{nwn}
\Phi_{\varepsilon\zeta}=\varepsilon  \Phi_{\varepsilon\zeta}^{\as}, \qquad
\Phi_{\varepsilon\zeta}\Phi=\varepsilon\zeta  (\Phi_{\varepsilon\zeta}\Phi)^{\as}.
\end{equation}

\begin{lemma}
\label{THEOREM 8}
Let ${\Phi}$ be a Frobenius block with the characteristic polynomial
\begin{equation}\label{xdn}
\chi_{\Phi}(x) =c_0+c_1x+\dots+c_{n-1}x^{n-1}+x^n
\end{equation}
 over a field\/ $\ff$
of characteristic different from $2$ with
involution. Let $\varepsilon,\zeta\in\{1,-1\}$ and $\varepsilon=1$ if the involution on $\ff$ is nonidentity.

{\rm(a)} If $\Phi$ is nonsingular, then ${\Phi}_{\varepsilon\zeta}$ exists if and only if
\begin{equation}\label{njnj}
(\varepsilon,\zeta)\ne (-1,1)\quad
\text{and}\quad p_{\Phi}(x) =
\zeta^{\deg(p_{\Psi})}\ba p_{\Phi}
(\zeta x).
\end{equation}
With these conditions
satisfied, we can take
\begin{equation}\label{dxdf}
{\Phi}_{\varepsilon\zeta}=
[\zeta^i a_{i+j}]_{i,j=1}^n
=\mat{\\[-26pt]a_2&a_3&\iddots&a_{n+1}\\
\zeta a_3&\iddots&\zeta a_{n+1}&\zeta a_{n+2}\\\iddots
&\iddots&\iddots&\iddots\\
\zeta^n a_{n+1}&\zeta^na_{n+2}&\iddots&\zeta^na_{2n}
},
\end{equation}
in which
\begin{equation}\label{nhb}
(a_2,\dots,a_{n+1}):=
\begin{cases}
 (1,0,0,\dots,0)&\text{if }\varepsilon =1,\\
(0,1,0,\dots,0)&\text{if }\varepsilon  =-1,
\end{cases}
\end{equation}
and
\begin{equation}\label{edc1}
a_{l+n}:=
-c_0a_{l} -
c_1a_{l+1}-\dots-c_{n-1}a_{l+n-1}
\qquad\text{for } l=2,\dots, n.
\end{equation}

{\rm(b)} If $\Phi$ is singular, then ${\Phi}_{\varepsilon\zeta}$ exists if and only if
\begin{equation}\label{eqk}
\text{$\varepsilon =1$ for $n$ odd, and $ \varepsilon=\zeta$ for $n$ even}.
\end{equation}
With these conditions
satisfied, we can take
${\Phi}_{\varepsilon\zeta}=  Z_n(\zeta)$, which is defined in \eqref{mmx}.
\end{lemma}

\begin{proof}
(a) Let $\Phi$ be nonsingular and let  ${\Phi}_{\varepsilon\zeta}$ exist. Then
 $(\varepsilon,\zeta)\ne (-1,1)$ according to \cite[Lemma
8]{ser_izv}.
By \eqref{nwn}, $$\Phi=
\varepsilon\zeta\Phi_{\varepsilon\zeta}^{-1}\Phi^{\as}
\Phi_{\varepsilon\zeta}^{\as}= \zeta  \Phi_{\varepsilon\zeta}^{-1}\Phi^{\as}
\Phi_{\varepsilon\zeta},$$ and so
\begin{equation}\label{eee}
\chi_{\Phi}(x) =\det(xI-\zeta
\Phi^{\as}) =\det(xI-\zeta
\ba\Phi)=\zeta^n\det(\zeta xI-
\ba\Phi)=
\zeta^n\ba\chi_{\Phi}
(\zeta x).
\end{equation}
Since $\chi_{\Phi}(x)=p_{\Phi}(x)^k$, we have $p_{\Phi}(x)^k=\zeta^n\ba p_{\Phi}(\zeta x)^k$.
The unique factorization property leads to $p_{\Phi}(x)=
\zeta^{\deg(p_{\Psi})}\ba p_{\Phi}(\zeta x)$, which proves \eqref{njnj}.

From \eqref{xdn} and \eqref{eee} we have
\begin{equation}\label{hqz}
c_0=\zeta^n\ba c_0,\quad
c_1=\zeta^{n+1}\ba c_1,\quad
\dots,\quad
c_{n-1}=\zeta^{2n-1}\ba c_{n-1}.
\end{equation}

Let us prove that the matrix
${\Phi}_{\varepsilon\zeta}$
defined in \eqref{dxdf} is
nonsingular. This follows from \eqref{nhb} and \eqref{edc1} if $\varepsilon =1$. Let
$\varepsilon =-1$. Then the involution on $\ff$ is the identity. From \eqref{njnj} we have  $\zeta=-1$.  By \eqref{eee}, $\chi_{\Phi}(x)\in\ff[x^2]$. Hence $c_1=0$ and  ${\Phi}_{\varepsilon\zeta}$ is nonsingular, which follows from \eqref{nhb} and \eqref{edc1}.

The equalities \eqref{edc1} imply that
$\Phi_{\varepsilon\zeta}\Phi =
[\zeta^ia_{i+j+1}]_{i,j=1}^n;$
and so the relations \eqref{nwn}
can be written in the form
\[
\zeta^ia_{i+j}=
\varepsilon \zeta^j\ba
a_{j+i},\qquad
\zeta^ia_{i+j+1}=\varepsilon
\zeta^{j+1}\ba
a_{j+i+1},
\]
i.e.,  in the form
\begin{equation}\label{ser30}
a_t= \varepsilon \zeta^t\ba
a_t\qquad\text{for }t=2,\dots, 2n+1.
\end{equation}
The relations \eqref{ser30} hold for all $t \le n+ 1$ since if $\varepsilon =-1$ then  $\zeta =-1$.
Let $l\in\{2,\dots, n+1\}$ and let \eqref{ser30} hold for all $t < n+ l$. Then the relations \eqref{ser30}
hold for $t = n + l$ since from \eqref{edc1} and \eqref{hqz} we have
\begin{align*}
a_{n+l}&=-c_0a_{l} -
c_1a_{l+1}-\dots-c_{n-1}a_{l+n-1}\\
&=-\zeta^n\ba c_0\cdot
\varepsilon \zeta^l\ba
a_l-\zeta^{n+1}\ba c_1\cdot
\varepsilon \zeta^{l+1}\ba
a_{l+1}
-\dots
-\zeta^{2n-1}\ba c_{n-1}\cdot
\varepsilon \zeta^{l+n-1}\ba
a_{l+n-1}
\\&=\varepsilon \zeta^{n+l}(-\ba c_0\ba a_l
-\ba c_1\ba a_{l+1}-\dots
-\ba c_{n-1}\ba a_{l+n-1})
=\varepsilon \zeta^{n+l}\ba
a_{n+l},
\end{align*}
which proves \eqref{ser30}. Hence, $\Phi_{\varepsilon\zeta}$ satisfies \eqref{nwn}.

(b)  Let $\Phi$ be singular. If ${\Phi}_{\varepsilon\zeta}$ exists, then \eqref{eqk} holds according to \cite[Lemma
8]{ser_izv}. The matrix $Z_n(\zeta)$ satisfies \eqref{nwn}.
\end{proof}

\section{Proofs of Corollaries \ref{ynh1} and \ref{theor}}
\label{vvz}

\begin{proof}[Proof of Corollary \ref{theor}]
 (A) We take $\mc O_{\ff}=\{J_n(\lambda)\,|\,\lambda\in\ff\}$.    The field  \eqref{aza} is $\ff$ with the identity involution. All forms \eqref{bhy} are equivalent to exactly one form $x_1^2+\dots+x_s^2$. Hence, each pair $A^{f(x)}_{\Phi}$ is isomorphic to exactly one direct sum of pairs $A^1_{\Phi}=(\Phi,
\Phi_{\varepsilon\zeta})$.
Let $\Phi=J_n(\lambda )$, and let   ${\Phi}_{\varepsilon\zeta}$ exist. If $\lambda \ne 0$, then by
\eqref{njnj}
$(\varepsilon,\zeta)\ne (-1,1)$ and $x-\lambda  =
x-\zeta\lambda $; thus $\zeta =1$, $(\varepsilon,\zeta)=(1,1)$; we take ${\Phi}_{11}=Z_n$.
If $\lambda = 0$, then
\eqref{eqk} holds and we take  ${\Phi}_{\varepsilon\zeta}=Z_n(\zeta)$.
Applying Theorem \ref{Theorem 6}, we obtain the summands (A).

(B) We take $\mc O_{\ff}=\{J_n(\lambda)\,|\,\lambda\in\ff\}$.  The field  \eqref{aza} is $\ff$ with nonidentity involution.  Each form \eqref{bhy} is equivalent to exactly one  form $-\ba x_1x_1-\dots-\ba x_lx_l+\ba x_{l+1}x_{l+1}
+\dots+\ba x_sx_s$. Hence, each pair $A^{f(x)}_{\Phi}$ is isomorphic to  a direct sum of pairs  $A^{\pm 1}_{\Phi}=(\Phi,\pm
\Phi_{11})$, determined uniquely  up to permutations of summands.  We take ${\Phi}_{11}=Z_n$. Applying Theorem \ref{Theorem 6}, we obtain the summands (B).

(C) Let
$\ff=\pp$ be a real
closed field,   and let
$\pp+i\pp$ be its
algebraic closure.
We take $\mc O_{\pp}=\{J_n(a),\: J_n(a+ib)^{\pp}\,|\,a,b\in\ff,\,b\ne 0\}$, in which $b$ is determined up to replacement by $-b$.

(c$_1$)  Let $\Phi=J_n(a )$ with $a \in \pp$, and let   ${\Phi}_{\varepsilon\zeta}$ exist. Then the field  \eqref{aza} is $\pp$, and so each form \eqref{bhy} is equivalent to exactly one  form $-x_1^2-\dots-x_l^2+x_{l+1}^2
+\dots+x_s^2$.
If $a\ne 0$, then by
\eqref{njnj}
$(\varepsilon,\zeta)\ne (-1,1)$ and $x-a  =
x-\zeta a $; thus $\zeta =1$, $(\varepsilon,\zeta)=(1,1)$, and we take ${\Phi}_{11}=Z_n$.
If $a=0$, then
\eqref{eqk} holds and we take  ${\Phi}_{\varepsilon\zeta}=Z_n(\zeta)$.

(c$_2$)  Let  $\Phi=J_n(a+ib )^{\pp}$ with $a,b \in \pp$, $b\ne 0$, and let   ${\Phi}_{\varepsilon\zeta}$ exist. If $\zeta =1$, then the field \eqref{aza} is $\pp+i\pp$ with the identity involution,
 $\varepsilon=1$  by
\eqref{njnj}, and we take ${\Phi}_{11}=Z_{2n}$.
If  $\zeta =-1$,
then the field \eqref{aza} is $\pp+i\pp$ with nonidentity involution,
 $p_{\Phi}(x)=
x^2-2ax+(a^2+b^2) $,  $a=0$ by
\eqref{njnj}, and we take ${\Phi}_{11}=K_n(\varepsilon)$.

Applying Theorem \ref{Theorem 6}, we obtain the summands (C).
\end{proof}

\begin{proof}[Proof of Corollary \ref{ynh1}]
Let $\ff=\cc$ with complex conjugation. By \eqref{rrd},
\begin{equation}\label{rdd}
\pm r(\ffff)=
\pm r(\cc^{\times}/\rr^{\times})= \{e\in\cc\,|\,|e|=1\}.
\end{equation}

(a)  This statement is obtained by applying Theorem \ref{yyj} and \eqref{rdd} to the canonical form of an isometric operator on a complex vector space with nondegenerate Hermitian form given in  \cite[Theorem 2.1(b)]{ser_iso}.

(b)  This statement is obtained by applying Theorem \ref{yyj} and \eqref{rdd} to the canonical form of a selfadjoint operator on a complex vector space with  nondegenerate Hermitian form given in Corollary \ref{theor}(B).
\end{proof}

\section*{Acknowledgements}
V. Futorny was supported by the CNPq (304467/2017-0) and the FAPESP (2018/23690-6).  V.V.~Sergeichuk wishes to thank the University of Sao Paulo, where the paper was written, for hospitality and the FAPESP for financial support (2018/24089-4).

\end{document}